\documentclass{amsart}
\usepackage{amstext,amssymb,amsthm,amsopn,newlfont,graphpap,graphics,graphicx,mathrsfs,enumitem}\usepackage{amstext,amssymb,amsthm,amsopn,newlfont,graphpap,graphics,graphicx,mathrsfs,mathtools,enumitem}
\usepackage[parfill]{parskip}
\usepackage[noadjust]{cite}
\usepackage{epigraph}
\usepackage[colorlinks=true,
            linkcolor=red,
            urlcolor=blue,
            citecolor=magenta]{hyperref}
\usepackage{color}
\usepackage{mathrsfs}
\allowdisplaybreaks
\theoremstyle{plain}
\newtheorem{thm}{Theorem}[section]
\newtheorem*{thm*}{Theorem}
\newtheorem{prop}{Proposition}[section]
\newtheorem*{prop*}{Proposition}

\newtheorem*{cor*}{Corollary}
\newtheorem{lem}{Lemma}[section]
\newtheorem*{lem*}{Lemma}
\theoremstyle{definition}

\newtheorem*{defn*}{Definition}

\newtheorem*{exmps*}{Examples}

\newtheorem*{exmp*}{Example}

\newtheorem*{exerc*}{Exercise}

\newtheorem{rems}{Remarks}[section]
\newtheorem*{rems*}{Remarks}

\newtheorem*{rem*}{Remark}
\newcommand{\N}{{\mathbb N}}
\newcommand{\Z}{{\mathbb Z}}
\newcommand{\R}{{\mathbb R}}
\newcommand{\C}{{\mathbb C}}
\newcommand{\F}{{\mathbb F}}

\renewcommand{\bar}{\overline}
\numberwithin{equation}{section}
\DeclareMathOperator{\Rep}{Re\ignorespaces}
\DeclareMathOperator{\Imp}{Im\ignorespaces}

\DeclareMathOperator{\orb}{orb}
\DeclareMathOperator{\supp}{supp} 
\DeclareMathOperator{\Per}{Per} 
\begin{document}
\title[On linear chaos in function spaces]
{On linear chaos in function spaces}
\author{John M. Jimenez}
\address{
Department of Mathematics\newline
De Anza College \newline
21250 Stevens Creek Blvd. \newline
Cupertino, CA 95014, USA
}
\email{jimenezmjohn@fhda.edu}
\author[Marat V. Markin]{Marat V. Markin}
\address{
Department of Mathematics\newline
California State University, Fresno\newline
5245 N. Backer Avenue, M/S PB 108\newline
Fresno, CA 93740-8001, USA
}
\email[corresponding author]{mmarkin@csufresno.edu}

\subjclass{Primary 47A16, 47B37, 47B38; Secondary 47A10}
\keywords{Hypercyclic vector, periodic point, hypercyclic operator, chaotic operator, spectrum}
\begin{abstract}
We show that, in $L_{p}(0,\infty)$ ($1\leq p <\infty$), bounded weighted translations as well as their unbounded counterparts are chaotic linear operators. We also extend the unbounded case to $C_{0}[0,\infty)$ and describe the spectra of the weighted translations provided the underlying spaces are complex.
\end{abstract}
\maketitle

\section[Introduction]{Introduction}

Extending the classical Rolewicz's result \cite{Rolewicz}
and the results of \cite{arXiv:1811.06640} for the sequence spaces $l_p$ ($1\leq p <\infty$), we show that, in the space $L_{p}(0,\infty)$ ($1\leq p <\infty$), the bounded weighted left translations
\[
(T_{w,a}x)(t)=wx(t+a)\quad (|w|>1,\ a>0)
\]
as well as their unbounded counterparts
\[
(T_{w,a}x)(t)=w^{t}x(t+a) (w>1,\  a>0)
\]
are chaotic linear operators (the latter forecasted in \cite[Remark $3.1$]{arXiv:1811.06640}). 

The chaoticity of the bounded weighted left translations in $C_{0}[0,\infty)$ established in \cite{Aron-Seoan-Weber2005}, we stretch the unbounded case from the sequence space $c_0$ \cite{arXiv:1811.06640} to the space $C_0[0,\infty)$ of real- or complex-valued functions continuous on $[0,\infty)$ and vanishing at infinity, which is Banach relative to
the norm
\[
C_0[0,\infty)\ni x\mapsto \|x\|_\infty:=\sup_{t\ge 0}|x(t)|
\]
(also forecasted in \cite[Remark $3.1$]{arXiv:1811.06640}) and describe the spectra of the weighted translations provided the underlying spaces are complex.

\section[Preliminaries]{Preliminaries}

\subsection{Hypercyclicity and Chaoticity}\

For a (bounded or unbounded) linear operator $T$ in a (real or complex) Banach space $X$, a nonzero vector 
\begin{equation*}
x\in C^\infty(T):=\bigcap_{n=0}^{\infty}D(T^n)
\end{equation*}
($D(\cdot)$ is the \textit{domain} of an operator, $T^0:=I$, $I$ is the \textit{identity operator} on $X$) is called \textit{hypercyclic} if its \textit{orbit} under $T$
\[
\orb(x,T):=\left\{T^nx\right\}_{n\in\Z_+}
\]
($\Z_+:=\left\{0,1,2,\dots\right\}$ is the set of \textit{nonnegative integers}) is dense in $X$.

Linear operators possessing hypercyclic vectors are said to be \textit{hypercyclic}.

If there exist an $N\in \N$ ($\N:=\left\{1,2,\dots\right\}$ is the set of \textit{natural numbers}) and a vector 
\[
x\in D\left(T^N\right)\quad \text{with}\quad T^Nx = x,
\]
such a vector is called a \textit{periodic point} for the operator $T$ of period $N$. If $x\ne 0$, we say that $N$ is a \textit{period} for $T$.

Hypercyclic linear operators with a dense in $X$ set $\Per(A)$ of periodic points are said to be \textit{chaotic}.

See \cite{Devaney,Godefroy-Shapiro1991,B-Ch-S2001}.

\begin{rems}\
\begin{itemize}
\item In the definition of hypercyclicity, the underlying space is necessarily
\textit{infini\-te-dimensional} and \textit{separable} (see, e.g., \cite{Grosse-Erdmann-Manguillot}).
\item For a hypercyclic linear operator $T$, the set $HC(T)$ of all its hypercyclic vectors
is necessarily dense in $X$, and hence, the more so, is the subspace $C^\infty(T)\supseteq HC(T)$.
\item Observe that
\[
\Per(A)=\bigcup_{N=1}^\infty \Per_N(A),
\]
where 
\[
\Per_N(A)=\ker(A^N-I),\ N\in \N
\]
is the \textit{subspace} of $N$-periodic points of $A$.
\end{itemize} 
\end{rems} 

Prior to \cite{B-Ch-S2001,deL-E-G-E2003}, the notions of linear hypercyclicity and chaoticity had been studied exclusively for \textit{continuous} linear operators on Fr\'echet spaces, in particular for \textit{bounded} linear operators on Banach spaces (for a comprehensive survey, see \cite{Bayart-Matheron,Grosse-Erdmann-Manguillot}). 

In \cite{Rolewicz}, S. Rolewicz provides the first example of \textit{hypercyclic} bounded linear operators on Banach spaces (see also \cite{Grosse-Erdmann-Manguillot}), which on the (real or complex) sequence space $l_p$ ($1\le p<\infty$) of $p$-summable sequences or $c_0$ of vanishing sequences, the latter equipped with the supremum norm
\[
c_0\ni x:=(x_k)_{k\in \N}\mapsto \|x\|_\infty:=\sup_{k\in \N}|x_k|,
\] 
are the weighted backward shifts
\[
T_w(x_k)_{k\in \N}:=w(x_{k+1})_{k\in \N}
\]
with $w\in \F$ ($\F:=\R$ or $\F:=\C$) such that $|w|>1$. Furthermore, Rolewicz's shifts are established to be \textit{chaotic} \cite{Godefroy-Shapiro1991}. 

In \cite{arXiv:1811.06640} (see also \cite{arXiv:1812.02294}), it is shown that the weighted backward shifts
\[
T_wx:=(w^k x_{k+1})_{k\in \N}
\]
with $w\in \F$ such that $|w|>1$ and maximal domain in the (real or complex) sequence spaces $l_p$ ($1\le p<\infty$) and $c_0$ are \textit{chaotic} unbounded linear operators and, provided the underlying space is complex, each $\lambda \in \C$ is a simple eigenvalue for $T_w$.

When establishing hypercyclicity, we obviate explicit construction of hypercyclic vectors by applying the subsequent version of the classical \textit{Birkhoff Transitivity Theorem} \cite[Theorem $1.16$]{Grosse-Erdmann-Manguillot} or the following \textit{Sufficient Condition for Hypercyclicity} {\cite[Theorem $2.1$]{B-Ch-S2001}}, which is an extension of \textit{Kitai's ctriterion} \cite{Kitai1982,Gethner-Shapiro1987}. 

\begin{thm}[Birkhoff Transitivity Theorem]\label{Birkhoff}\ \\
A bounded linear operator $T$ on a (real or complex) infinite-dimensional separable Banach space $X$ is hypercyclic iff it is topologically transitive, i.e., for any nonempty open subsets $U$ and $V$ of $X$, there exists an $n\in \Z_+$ such that 
\[
T^n(U)\cap V \neq \emptyset.
\]
\end{thm}

Cf. \cite[Theorem $2.19$]{Grosse-Erdmann-Manguillot}.

\begin{thm}[Sufficient Condition for Hypercyclicity]\label{SCH}\ \\
Let $X$ be a  (real or complex) infinite-dimensional separable Banach space and $T$ be a densely defined linear operator in $X$ such that each power $T^{n}$ ($n\in\N$) is a closed operator. If there exists a set
\[
Y\subseteq C^\infty(T):=\bigcap_{n=1}^\infty D(T^n)
\]
dense in $X$ and a mapping $S:Y\to Y$ such that
\begin{enumerate}
\item $\forall\, x\in Y:\ TSx=x$ and
\item $\forall\, x\in Y:\ T^nx,S^nx\to 0,\ n\to \infty$,
\end{enumerate}
then the operator $T$ is hypercyclic.
\end{thm}

\subsection{Resolvent Set and Spectrum}\

For a linear operator $T$ in a complex Banach space $X$, the set
\[
\rho(A):=\left\{ \lambda\in \C \,\middle|\, \exists\, (T-\lambda I)^{-1}\in L(X) \right\}
\]
($L(X)$ is the space of bounded linear operators on $X$) and its complement $\sigma(T):=\C\setminus \rho(T)$ are called the \textit{resolvent set} and the \textit{spectrum} of $T$, respectively.

The spectrum $\sigma(T)$ of a closed linear operator $T$ in a complex Banach space $X$ is the union of the following pairwise disjoint sets:
\begin{equation*}
\begin{split}
& \sigma_p(T):=\left\{\lambda\in \C \,\middle|\,T-\lambda I\ \text{is \textit{not injective}, i.e., $\lambda$ is an \textit{eigenvalue} of $T$} \right\},\\
& \sigma_c(T):=\left\{\lambda\in \C \,\middle|\,T-\lambda I\ \text{is \textit{injective},
\textit{not surjective}, and $\overline{R(T-\lambda I)}=X$} \right\},\\
& \sigma_r(T):=\left\{\lambda\in \C \,\middle|\,T-\lambda I\ \text{is \textit{injective} and $\overline{R(T-\lambda I)}\neq X$} \right\}
\end{split}
\end{equation*}
($R(\cdot)$ is the \textit{range} of an operator, and $\overline{\cdot}$ is the \textit{closure} of a set), called the \textit{point}, \textit{continuous} and \textit{residual spectrum} of $T$, respectively (see, e.g., \cite{Dun-SchI,Markin2020EOT}).

\section{Bounded Weighted Translations on $L_{p}(0,\infty)$)}

\begin{thm}[Bounded Weighted Translations on $L_p(0,\infty)$]\label{Lp}\ \\
On the (real or complex) space $L_p(0,\infty)$ ($1\le p<\infty$), the weighted left translation
\begin{equation*}
(T_{w,a}x)(t):=wx(t+a),\ x\in L_p(0,\infty),\, t\ge 0,
\end{equation*}
with $w\in \F$ such that $|w|>1$ and $a>0$ is a chaotic bounded linear operator. 

Furthermore, provided the underlying space is complex,
\begin{equation}\label{sp}
\sigma(T_{w,a})=\left\{\lambda\in \C\,\middle|\, |\lambda|\le |w| \right\}
\end{equation}
with
\begin{equation}\label{pcsp}
\sigma_p(T_{w,a})=\left\{\lambda\in \C\,\middle|\, |\lambda|< |w| \right\}\quad 
\text{and}\quad
\sigma_c(T_{w,a})=\left\{\lambda\in \C\,\middle|\, |\lambda|= |w| \right\}.
\end{equation}
\end{thm}
    
\begin{proof}
Let $1\le p<\infty$, $w\in \F$ such that $|w|>1$, and $a>0$ be arbitrary and, for the simplicity of notation, let $T:=T_{w,a}$.

The \textit{linearity} of $T$ is obvious. Its \textit{boundedness} immediately follows from the fact that 
\[
T=wB,
\]
where 
\begin{equation*}
(Bx)(t):=x(t+a),\ x\in L_p(0,\infty),\, t\ge 0,
\end{equation*}
is a left translation operator with $\|B\|=1$, and hence,
\begin{equation}\label{ON}
\|T\|=|w|\|B\|=|w|
\end{equation} 
(here and wherever appropriate, $\|\cdot\|$ also stands for the \textit{operator norm}).

Suppose that
\[
U,V\subseteq L_{p}(0,\infty)
\]
are arbitrary nonempty open sets. 

By the denseness in $L_{p}(0,\infty)$ of the equivalence classes represented by $p$-integrable on $(0,\infty)$ eventually zero functions (see, e.g., \cite{MarkinRA}), there exist equivalence classes 
\[
x\in U\quad \text{and}\quad y\in V
\]
represented by such functions $x(\cdot)$ and $y(\cdot)$, respectively. Since the representative functions are eventually zero, 
\[
\exists\, N\in \mathbb{N}\ \forall\, t>Na:\ x(t)=0\ \text{and}\ y(t)=0.
\]

    For an arbitrary $ n \geq N$, the $p$-integrable on $(0,\infty)$ eventually zero function
    
 \[    
z_n(t) := \begin{cases} 
      x(t), & t\in [0,Na), \\
      w^{-n}y(t-an), &  t\in [na,Na+na), \\
      0, & \text{otherwise},
   \end{cases}
\]
    
represents an equivalence class $z_n\in L_{p}(0,\infty)$.

Observe that, for all $n\ge N$, 
    \[   ( T^nz_n)(t)=
      y(t),\ t\ge 0,
    \]
    and 
    \[
    \|z_n-x\|_{p}=|w|^{-n}\|y\|_p\to 0,\ n\to \infty
    \]
    
Hence, for all sufficiently large $n\in \N$, 
\[
z_n\in U\quad \text{and} \quad T^{n}z_n=y\in V
\]
(see Figure \ref{fig5}).

\begin{figure}[!htbp]
\begin{center}
\includegraphics[height=1.5in]{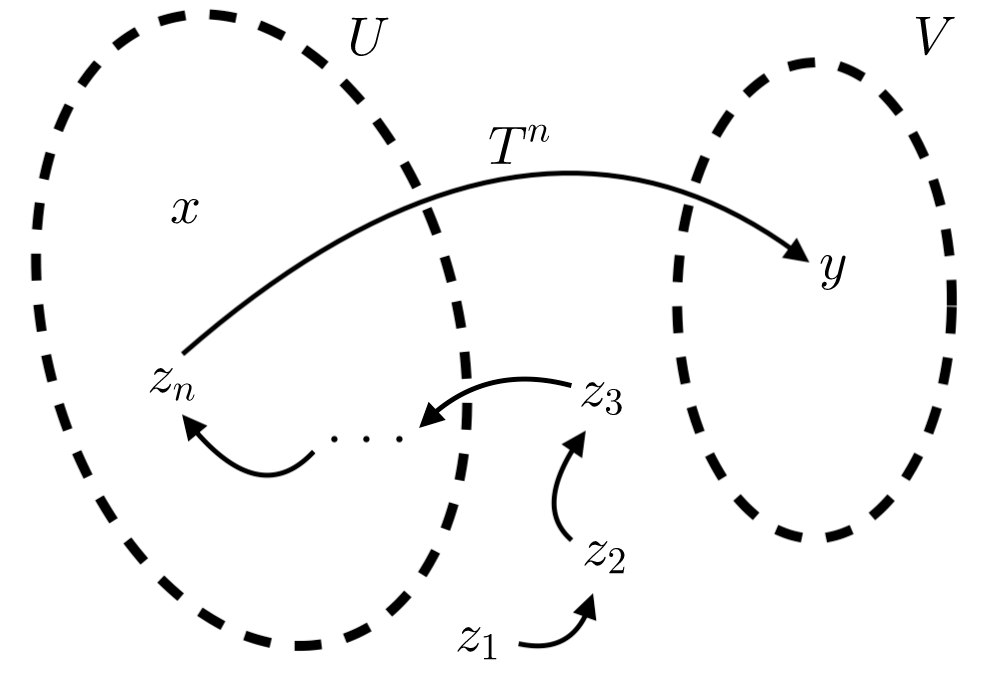}
\end{center}
\caption{}
\label{fig5}
\end{figure} 
    
By the \textit{Birkhoff Transitivity Theorem} (Theorem \ref{Birkhoff}), we infer that the operator $T$ is \textit{hypercyclic}.

To prove that $T$ has a dense set of periodic points, let us first show that each $N\in \N$ is a period for $T$.

For an arbitrary $N\in\N$, let 
\[
x\in \ker T^N\setminus \left\{ 0\right\},
\]
where
\[
\ker T^N=\left\{f\in L_p(0,\infty)\,\middle|\,  f(t)=0,\ t>Na\right\}.
\]

Then the $p$-integrable on $(0,\infty)$ function
\begin{equation}\label{pN}
 x_N(t) := w^{-kN}x(t-kNa),\ t\in D_{k}:=[kNa,(k+1)Na),k\in \Z_+,
\end{equation}
represents an $N$-periodic point $x_N$ of $T$.

Indeed, in view of $|w|>1$,
\begin{align*}
\int_0^\infty\left|x_N(t)\right|^p\,dt&=\sum_{k=0}^{\infty}  \displaystyle\int_{D_{k}} \left| w^{-kN}x(t-kNa) \right|^{p}\,dt = \sum_{k=0}^{\infty}{\left(|w|^{-pN}\right)}^k\int_{0}^{Na} \left|x(t) \right|^{p}\,dt\\
&=\sum_{k=0}^{\infty}{\left(|w|^{-pN}\right)}^k\|x\|_p^p= \frac{1}{1-|w|^{-pN}}\|x\|_p^p<\infty,
\end{align*}
and hence, $x_N\in L_{p}(0,\infty)$. 

Further, since
\begin{align*}
(T^Nx_N)(t)&=w^Nx_N(t+Na)=w^N w^{-kN}x(t+Na-kNa)\\
&=w^{-(k-1)N}y(t-(k-1)Na),\ t\in D_{k-1},k\in \N,
\end{align*}
we infer that
\[
T^Nx_N=x_N.
\]

Suppose that $x\in L_{p}(0,\infty)$ is an arbitrary equivalence class represented by a $p$-integrable on $(0,\infty)$ eventually zero function $x(\cdot)$. Then 
\[
\exists\, M\in \N:\ x(t)=0,\ t>Ma.
\]

Let $x_N$ be the periodic point of the operator $T$ of an arbitrary period $N\ge M$ defined based on $x$ by \eqref{pN}. Then
\begin{align*}
\|x_N-x\|_{p}^p&=\sum_{k=0}^{\infty} \int_{D_{k}} \left|x_N(t)-x(t)\right|^{p}\,dt
=\sum_{k=1}^{\infty} \int_{D_{k}} \left|w^{-kN}x(t-kNa)\right|^{p}\,dt\\
&=\sum_{k=1}^{\infty}{\left(|w|^{-pN}\right)}^k\int_{0}^{Na} \left|x(t) \right|^{p}\,dt
=\sum_{k=1}^{\infty}{\left(|w|^{-pN}\right)}^k\|x\|_p^p\\
&= \frac{|w|^{-pN}}{1-|w|^{-pN}}\|x\|_p^p\to 0,\ N\to \infty. 
\end{align*}

By the denseness in $L_{p}(0,\infty)$ ($1\le p<\infty$) of the subspace 
\begin{equation}\label{YLp}
Y:=\bigcup_{n=1}^\infty \ker T^n,
\end{equation}
where
\begin{equation}\label{kerTnLp}
\ker T^n=\left\{f\in L_p(0,\infty)\,\middle|\,  f(t)=0,\ t>na\right\},\ n\in \N,
\end{equation}
of the equivalence classes represented by $p$-integrable on $(0,\infty)$ eventually zero functions, we infer that the set $\Per(T)$ of periodic points of $T$ is \textit{dense} in $L_{p}(0,\infty)$ as well, and hence, the operator $T$ is \textit{chaotic}.

Now, assuming that the space $L_{p}(0,\infty)$ is complex, let us prove \eqref{sp} and \eqref{pcsp}.

In view of \eqref{ON}, by \textit{Gelfand's Spectral Radius Theorem} \cite{Markin2020EOT},
\begin{equation}\label{spincl}
\sigma(T)\subseteq \left\{\lambda\in \C\,  \middle |\, |\lambda|\leq |w| \right\}.
\end{equation}

For an arbitrary $\lambda\in \mathbb{C}$ with $|\lambda|<|w|$, let 
\[
x\in \ker T\setminus \left\{ 0\right\}\subseteq Y\setminus \left\{ 0\right\},
\]
where
\[
\ker T=\left\{f\in L_p(0,\infty)\,\middle|\,  f(t)=0,\ t>a\right\}
\]
(see \eqref{kerTnLp}).

Then the $p$-integrable on $(0,\infty)$ function
\begin{equation}\label{efbounded}
x_{\lambda}(t):=\left(\dfrac{\lambda}{w}\right)^{k}x(t-ka),\
t\in [ka,(k+1)a),k\in \Z_+,\quad (0^0:=1)
\end{equation}
is an eigenvector for $T$ associated with $\lambda$.

Indeed, in view of $|\lambda|<|w|$,
\begin{align*}
0<\|x_{\lambda}\|_p^p&=\int_0^\infty|x_{\lambda}(t)|^p\,dt=\sum_{k=0}^\infty \int_{ka}^{(k+1)a}\left|\left(\dfrac{\lambda}{w}\right)^{k}x(t-ka)\right|^p\,dt
\\
&=\sum_{k=0}^\infty {\left|\dfrac{\lambda}{w}\right|}^{kp}\int_{ka}^{(k+1)a}|x(t-ka)|^p\,dt 
=\sum_{k=0}^\infty {\left({\left|\dfrac{\lambda}{w}\right|}^{p}\right)}^{k}\int_{0}^{a}|x(t)|^p\,dt
\\
&=\sum_{k=0}^\infty {\left({\left|\dfrac{\lambda}{w}\right|}^{p}\right)}^{k}\|x\|_p^p<\infty,
\end{align*}
and hence, $x_{\lambda}\in L_{p}(0,\infty)\setminus \left\{ 0\right\}$. 

Further,
\begin{align*}
(Tx_{\lambda})(t)&=wx_{\lambda}(t+a)=w\left(\dfrac{\lambda}{w}\right)^{k}x(t+a-ka)\\
&=\lambda\left(\dfrac{\lambda}{w}\right)^{k-1}x(t-(k-1)a),\ t\in [(k-1)a,ka),k\in \N,
\end{align*}
which implies that
\begin{equation}\label{EV}
Tx_{\lambda}=\lambda x_{\lambda},
\end{equation}
and hence, $\lambda\in \sigma_p(T)$.

Conversely, let $\lambda\in \sigma_p(T)$ be an arbitrary eigenvalue for $T$ with an associated eigenvector $x_{\lambda}\in L_p(0,\infty)\setminus \left\{0 \right\}$. Then, for
\[
x_k(t):=x_{\lambda}(t),\ t\in [ka,(k+1)a), k\in \Z_+,
\]
by \eqref{EV}, we have:
\[
\lambda x_{k-1}(t)=wx_{k}(t+a),\ \text{on}\ [ka,(k+1)a)\pmod{\lambda_1},\ k\in \N,
\]
($\lambda_1$ is the Lebesgue measure on $\R$).

Whence, 
\[
x_k(t)=\left(\dfrac{\lambda}{w}\right)^{k}x_{\lambda}(t-ka),\ \text{on}\ [ka,(k+1)a)\pmod{\lambda_1},\ k\in \Z_+,
\]
which, in view of $x_{\lambda}\ne 0$, implies that
\[
0<\int_{0}^{a}|x_{\lambda}(t)|^p\,dt \le \int_0^\infty|x_{\lambda}(t)|^p\,dt=\|x_{\lambda}\|_p^p<\infty
\]
and
\begin{align*}
\infty>\|x_{\lambda}\|_p^p&=\int_0^\infty|x_{\lambda}(t)|^p\,dt
=\sum_{k=0}^\infty \int_{ka}^{(k+1)a}\left|\left(\dfrac{\lambda}{w}\right)^{k}x_{\lambda}(t-ka)\right|^p\,dt
\\
&=\sum_{k=0}^\infty {\left|\dfrac{\lambda}{w}\right|}^{kp}\int_{ka}^{(k+1)a}|x_{\lambda}(t-ka)|^p\,dt 
=\sum_{k=0}^\infty {\left({\left|\dfrac{\lambda}{w}\right|}^{p}\right)}^{k}\int_{0}^{a}|x_{\lambda}(t)|^p\,dt.
\end{align*}

The convergence of the latter series implies that
\[
{\left({\left|\dfrac{\lambda}{w}\right|}^{p}\right)}^{k}\to 0,\ k\to\infty,
\]
which, in its turn, means that 
\[
|\lambda|<|w|.
\]

Thus, $x_\lambda$ can be represented by a $p$-integrable on $(0,\infty)$ function $x_{\lambda}(\cdot)$ of the form given by \eqref{efbounded}, where the corresponding $x\in \ker T\setminus \left\{ 0\right\}$ is represented by
\[
x(t):=\chi_{[0,a]}(t)x_{\lambda}(t),\ t\ge 0
\]
($\chi_\delta(\cdot)$ is the \textit{characteristic function} of a set $\delta$).

The above proves that
\begin{equation}\label{psp}
\sigma_p(T)=\left\{\lambda\in \C\,  \middle |\, |\lambda|<|w| \right\}.
\end{equation}

Considering that $\sigma(T)$ is a \textit{closed} set in $\C$ (see, e.g., \cite{Dun-SchI,Markin2020EOT}), we infer from \eqref{spincl} and \eqref{psp} that \eqref{sp} holds.

Since, by \cite[Lemma $2.53$]{Grosse-Erdmann-Manguillot}, the hypercyclicity of $T$ implies the operator $T-\lambda I$ has a \textit{dense range} for all $\lambda\in \C$, we infer that 
\[
\sigma_{r}(T)=\emptyset
\]
(cf. \cite[Proposition $4.1$]{arXiv:2106.14872}, \cite[Lemma $1$]{MarkSich2019(1)}), and hence, in view of \eqref{sp}
and \eqref{psp}, we conclude that
\[
\sigma_c(T)=\left\{\lambda\in \C\,  \middle |\, |\lambda|=|w| \right\}.
\] 

Thus, \eqref{pcsp} holds as well.
\end{proof}

\section{Unbounded Weighted Translations in $L_{p}(0,\infty)$)}

\begin{lem}[Closedness of Powers]\label{CPLp}\ \\
In the (real or complex) space $L_p(0,\infty)$ ($1\le p<\infty$), for the weighted left translation
\begin{equation*}
(T_{w,a}x)(t):=w^{t} x(t+a),\ t\ge 0,
\end{equation*}
with $w>1$, $a>0$, and domain 
\begin{equation*}
D(T_{w,a}):=\left\{ x\in L_p(0,\infty)\, \middle |\, \displaystyle\int_{0}^{\infty} |w^{t} x(t+a)|^{p}\, dt <\infty \right\},
\end{equation*}
each power $T_{w,a}^n$ ($n\in \N$) is a densely defined unbounded closed linear operator.
\end{lem}

\begin{proof}
Let $1\le p<\infty$,  $w>1$, $a>0$, and $n\in \N$ be arbitrary and, for the simplicity of notation, let $T:=T_{w,a}$.

The \textit{linearity} of $T$ is obvious and implies that for $T^n$.

Inductively,
\begin{equation}\label{Tn}
(T^nx)(t)=w^tw^{t+a}\dots w^{t+(n-1)a}x(t+na)
=w^{nt+\frac{(n-1)na}{2}}x(t+na),\ t\ge 0,
\end{equation}
and
\begin{equation}\label{DTn}
D(T^n)=\left\{ x\in L_p(0,\infty)\, \middle |\, \displaystyle \int_{0}^{\infty} {\left|w^{nt+\frac{(n-1)na}{2}}x(t+na)\right|}^{p}\, dt <\infty \right\}.
\end{equation}

By the denseness in $L_{p}(0,\infty)$ ($1\le p<\infty$) of the subspace 
\begin{equation}\label{YLpU}
Y:=\bigcup_{m=1}^\infty \ker T^m,
\end{equation}
where
\begin{equation}\label{kerTnLpU}
\ker T^m=\left\{f\in L_p(0,\infty)\,\middle|\,  f(t)=0,\ t>ma\right\},\ m\in \N,
\end{equation}
of the equivalence classes represented by $p$-integrable on $(0,\infty)$ eventually zero functions and the inclusion
\begin{equation}\label{YLpCinf}
Y\subset C^\infty(T):=\bigcap_{m=1}^\infty D(T^m),
\end{equation}
which follows from \eqref{DTn}, we infer that the operator $T^n$ is \textit{densely defined}.

The  \textit{unboundedness} of $T^n$ follows from the fact that, for the equivalence classes $e_m\in L_p(0,\infty)$, $m \in \N$, represented by
\[
e_m(t):=\chi_{[m,m+1]}(t),\ m\in \N,t\ge 0,
\]
we have:
\[
e_m\in D(T^n),\ \|e_m\|_p=1,\ m\in \N,
\]
and, for all $m\in \N$ sufficiently large so that $m\ge na$, in view of $w>1$, 
\begin{align*}
\|T^ne_m\|_{p}
&=\left[\int_{0}^{\infty} {\left|w^{nt+\frac{(n-1)na}{2}}e_m(t+na)\right|}^{p}\, dt\right]^{1/p}\\
&=
\left[\int_{m-na}^{m+1-na} w^{p\left(nt+\frac{(n-1)na}{2}\right)}\,dt\right]^{1/p}
\\
&\ge w^{n(m-na)+\frac{(n-1)na}{2}}\to \infty,\ m\to\infty.
\end{align*}

Let a sequence $(x_m)_{m\in \N}$ in $L_p(0,\infty)$ be such that 
\[
D(T^n)\ni x_m\to x\in L_p(0,\infty),\ m\to \infty,
\]
and
\[
T^nx_m\to y\in L_p(0,\infty),\ m\to \infty.
\]

The sequences $(x_m(\cdot))_{m\in \N}$ and
$\left((T^nx_m)(\cdot)\right)_{m\in \N}$ of the $p$-integrable on $(0,\infty)$ representatives of the corresponding equivalence classes converging in $p$-norm on $(0,\infty)$, also converge in the Lebesgue measure $\lambda_1$ on $(0,\infty)$, and hence, by the \textit{Riesz theorem} (see, e.g., \cite{MarkinRA}), there exist subsequences $(x_{m(k)}(\cdot))_{k\in \N}$ and $\left((T^nx_{m(k)})(\cdot)\right)_{k\in \N}$ convergent \textit{a.e.} on $(0,\infty)$ relative to $\lambda_1$, i.e.,
\begin{equation}\label{mod1}
x_{m(k)}(t)\to x(t)\ \text{on}\ (0,\infty) \pmod{\lambda_1}
\end{equation}
and
\begin{equation}\label{mod2}
(T^nx_{m(k)})(t)\to y(t)\ \text{on}\ (0,\infty) \pmod{\lambda_1}.
\end{equation}

By \eqref{mod1},
\begin{align*}
(T^nx_{m(k)})(t)&=w^{nt+\frac{(n-1)na}{2}}x_{m(k)}(t+na)\\
&\to w^{nt+\frac{(n-1)na}{2}}x(t+na)\ \text{on}\ (0,\infty) \pmod{\lambda_1},
\end{align*}
which by \eqref{mod2}, in view of the \textit{completeness} of the \textit{Lebesgue} measure (see, e.g., \cite{MarkinRA}), implies that
\[
w^{nt+\frac{n(n-1)a}{2}}x(t+na)=y(t)\pmod{\lambda_1},
\]
and hence,
\[
x\in D(T^n)\quad \text{and}\quad T^nx=y.
\]

By the \textit{Sequential Characterization of Closed Linear Operators} (see, e.g., \cite{Markin2020EOT}) the operator $T^n$ is \textit{closed}.
\end{proof}

\begin{thm}[Unbounded Weighted Translations in $L_p(0,\infty)$]\label{LpU}\ \\
In the (real or complex) space $L_p(0,\infty)$ ($1\le p<\infty$), the weighted left translation
\begin{equation*}
(T_{w,a}x)(t):=w^{t} x(t+a),\ t\ge 0,
\end{equation*}
with $w>1$, $a>0$, and domain 
 \begin{equation*}
 D(T_{w,a}):=\left\{ x\in L_p(0,\infty)\, \middle |\, \displaystyle\int_{0}^{\infty} |w^{t} x(t+a)|^{p}\, dt <\infty \right\}
 \end{equation*}
is a chaotic unbounded linear operator. 

Furthermore, provided the underlying space is complex,
\begin{equation}\label{sp1}
\sigma(T_{w,a})=\sigma_p(T_{w,a})=\mathbb{C}.
\end{equation}
\end{thm}
    
\begin{proof} 
Let $1\le p<\infty$,  $w>1$, and $a>0$ be arbitrary and, for the simplicity of notation, let $T:=T_{w,a}$.


For the \textit{dense} in $L_p(0,\infty)$ subspace $Y$ of the equivalence classes  represented by $p$-integrable eventually zero functions
(see \eqref{YLpU} and \eqref{kerTnLpU}), we have inclusion \eqref{YLpCinf}.

The mapping 
\[
Y\ni x\mapsto Sx\in Y,
\]
where the equivalence class $Sx$ is represented by
\begin{equation}\label{S}
(Sx)(t):=\begin{cases}
w^{-(t-a)}x(t-a), &t>a,\\ 
0, & \text{otherwise},
\end{cases}
\end{equation}
is well defined since the function $(Sx)(\cdot)$ is eventually zero and, in view of $w>1$,
\begin{align*}
 \int_{0}^{\infty}\left|(Sx)(t)\right|^p\,dt&= \int_{a}^{\infty}\left|w^{-(t-a)p}x(t-a)\right|^p\,dt =\int_{0}^{\infty} w^{-tp}|x(t)|^p\,dt  
 \\
&\le \int_{0}^{\infty} |x(t)|^p\,dt<\infty.
\end{align*}

As is easily seen,
\begin{equation}\label{RI}
\forall\, x\in Y:\ TSx=x.
\end{equation}

Let $x\in  Y$, represented by a $p$-integrable on $(0,\infty)$ eventually zero function $x(\cdot)$, be arbitrary. Then 
\[
\exists\, M\in \mathbb{N}:\ \supp x:=\bar{\left\{t\in (0,\infty)\,\middle|\, x(t)\neq 0\right\}}\subseteq [0,Ma].
\]

By \eqref{Tn},
\begin{equation*}
\forall\, n\ge M:\ T^nx=0,
\end{equation*}
and hence,
\[
T^nx\to 0,\ n\to \infty.
\]

Based on \eqref{S}, inductively,
\begin{equation}\label{Sn}
\begin{aligned}
(S^nx)(t)
&=
\begin{cases}
0, & 0\le t<na,\\
w^{-(t-a)}w^{-(t-2a)}\dots w^{-(t-na)}x(t-na), &t\ge na,
\end{cases}
\\
&=
\begin{cases}
0, & 0\le t<na,\\
w^{-nt+\frac{n(n+1)a}{2}}x(t-na), &t\ge na,\\ 
\end{cases}
\quad x\in Y,n\in \N.
\end{aligned}
\end{equation}

In view of $w>1$, we have:
\begin{align*}
\|S^nx\|_{p}&=
{\left[\int_{0}^{\infty}\left|(S^{n}x)(t)\right|^{p}\,dt \right]}^{1/p}
= {\left[\int_{na}^{\infty} \left|w^{ -nt+\frac{n(n+1)a}{2}}x(t-na)\right|^p\, dt\right]}^{1/p}
\\
&\leq w^{ -n\cdot na+\frac{n(n+1)a}{2}} {\left[\int_{na}^{\infty} \left|x(t-na)\right|^p\, dt\right]}^{1/p}
\\
&= w^{-\frac{n(n-1)a}{2}}{\left[\int_{0}^{\infty} |x(t)|^p\, dt\right]}^{1/p}= w^{-\frac{n(n-1)a}{2}} \|x\|_{p},\ x\in Y,n\in \N.
\end{align*}

Whence, since $w>1$ and $a>0$, we deduce that 
\begin{equation*}
\forall\, x\in Y:\ \lim_{n\to \infty}{\|S^nx\|_p}^{1/n}=0,
\end{equation*}
or equivalently,
\begin{equation}\label{GS}
\forall\, x\in Y,\ \forall\, \alpha\in (0,1)\
\exists\, c=c(x,\alpha)>0\ \forall\, n\in \N:\ \|S^nx\|_{p}\le c\alpha^n \|x\|_{p},
\end{equation}
which implies
\[
\forall\, x\in Y:\ S^nx\to 0,\ n\to \infty.
\]

From the above and the fact that, by the \textit{Closedness of Powers Lemma} (Lemma \ref{CPLp}), each power $T^n$ ($n\in \N$) is a \textit{closed operator}, by the \textit{Sufficient Condition for Hypercyclicity} (Theorem \ref{SCH}), we infer that the operator $T$ is  \textit{hypercyclic}.

To prove that $T$ has a dense set of periodic points, let us first show that each $N\in \N$ is a period for $T$.

Let $N\in\N$ and
\begin{equation}\label{TN}
x\in \ker T^N\setminus \left\{ 0\right\}\subseteq Y\setminus \left\{ 0\right\},
\end{equation}
where
\[
\ker T^N=\left\{f\in L_p(0,\infty)\,\middle|\,  f(t)=0,\ t>Na\right\},
\]
be arbitrary.

By estimate \eqref{GS}
\begin{equation}\label{pN1}
x_N:=\sum_{k=0}^\infty S^{kN}x\in L_p(0,\infty)
\end{equation}
is well defined and, in view of \eqref{Sn}, is represented by the $p$-integrable on $(0,\infty)$ function
\begin{equation*}
x_N(t) := w^{-kNt +\frac{kN(kN+1)a}{2}}x(t-kNa),\ t\in D_k:= [kNa,(k+1)Na),k\in \Z_+.
\end{equation*}

Since, in view of \eqref{TN} and \eqref{RI}, 
\[
\sum_{k=0}^{\infty}T^{N}S^{kN}x=\sum_{k=1}^{\infty}S^{(k-1)N}x=x_N,
\]
by the \textit{closedness} of the operator $T^{N}$, we infer that
\[
x_N\in D(T^{N})\quad \text{and}\quad T^{N}x_N=x_N,
\]
(see, e.g., \cite{Markin2020EOT}), and hence, $x_N$ is an $N$-periodic point for $T$.

Suppose that $x\in Y$ is an arbitrary equivalence class represented by a $p$-integrable on $(0,\infty)$ eventually zero function $x(\cdot)$. Then 
\[
\exists\, M\in \N:\ x(t)=0,\ t>Ma.
\]

Then, for an arbitrary period $N\ge M$, \eqref{TN} holds and there exists
an $N$-periodic point $x_N$ for the operator $T$ defined based on $x$ by \eqref{pN1}. By estimate \eqref{GS},
\begin{align*}
\left\|x_N-x\right\|
&=\left\|\sum_{k=1}^\infty S^{kN}x\right\|_p
\le \sum_{k=1}^\infty  \left\|S^{kN}x\right\|_p
\le c\sum_{k=1}^\infty {\left(\alpha^N\right)}^k \|x\|_{p}\\
&=c\frac{\alpha^N}{1-\alpha^N}\|x\|_{p}\to 0,\ N\to \infty.
\end{align*}

Whence, in view of the denseness of $Y$ in $L_{p}(0,\infty)$, we infer that the set $\Per(T)$ of periodic points of $T$ is dense in $L_{p}(0,\infty)$ as well, and hence, the operator $T$ is \textit{chaotic}.

Now, assuming that the space $L_{p}(0,\infty)$ is complex, let us prove \eqref{sp1}.

Let $\lambda\in \mathbb{C}$ and 
\begin{equation}\label{T}
x\in \ker T\setminus \left\{ 0\right\}\subseteq Y\setminus \left\{ 0\right\},
\end{equation}
where
\[
\ker T=\left\{f\in L_p(0,\infty)\,\middle|\,  f(t)=0,\ t>a\right\},
\]
be arbitrary.

By estimate \eqref{GS}, for 
\[
\alpha:=(|\lambda|+1)^{-1}\in (0,1),
\]
we have:
\begin{equation}\label{GSS}
\exists\, c=c(x,\alpha)>0\ \forall\, k\in \N:\
\|\lambda^kS^kx\|_p\le |\lambda|^kc\alpha^k\|x\|_{p}=c{\left(|\lambda|\alpha\right)}^k\|x\|_{p},
\end{equation}
where $0\le |\lambda|\alpha^k=|\lambda|(|\lambda|+1)^{-k}<1$.

By estimate \eqref{GSS},
\begin{equation*}
x_\lambda:=\sum_{k=0}^\infty \lambda^kS^{k}x\in L_p(0,\infty)
\end{equation*}
is well defined and, in view of \eqref{Sn}, is represented by the $p$-integrable on $(0,\infty)$ function
\begin{equation*}
x_\lambda(t):=\lambda^{k}w^{-kt+\frac{k(k+1)a}{2}}x(t-ka),\ t\in [ka,(k+1)a),k\in \Z_+,\quad (0^0:=1).
\end{equation*}

Since
\[
\|x_\lambda\|_p^p=\int_{0}^{\infty} |x_\lambda(t)|^{p}\,dt
\ge \int_{0}^a |x_\lambda(t)|^{p}\,dt= \int_{0}^a |x(t)|^{p}\,dt
= \int_{0}^\infty |x(t)|^{p}\,dt=\|x\|_p^p>0,
\]
we infer that $x_\lambda\neq 0$.

Further, since, in view of \eqref{T} and \eqref{RI}, 
\[
\sum_{k=0}^{\infty}T(\lambda^k S^{k}x)
=\lambda \sum_{k=1}^{\infty}\lambda^{k-1}S^{k-1}=\lambda x_\lambda,
\]
by the \textit{closedness} of the operator $T$, we conclude that
\[
x_\lambda\in D(T)\quad \text{and}\quad Tx_\lambda=\lambda x_\lambda,
\]
(see, e.g., \cite{Markin2020EOT}).

Thus, $\lambda\in \sigma_p(T)$ and $x_{\lambda}$ is an eigenvector of $T$ associated with $\lambda$, which proves \eqref{sp1}. 
\end{proof}

\section{Bounded Weighted Translations on $C_0[0,\infty)$}

In \cite[Theorem $2.3$]{Aron-Seoan-Weber2005}, it is shown that, on the (real or complex) space $C_0[0,\infty)$, the bounded linear weighted left translation operator
\begin{equation*}
(T_{w,a}x)(t):=wx(t+a),\ t\ge 0,
\end{equation*}
with $|w|>1$ and $a>0$ is \textit{chaotic} and
\[
\left\{\lambda\in \C\,\middle|\, 0<|\lambda|< |w| \right\}\subseteq \sigma_p(T)
\]
based on the simple fact that, for each $\lambda \in \C$, $\Rep\lambda<0$, 
the equation
\[
T_{w,a}x=we^{a\lambda}x
\]
is satisfied by the function
\[
x(t):=e^{\lambda t},\ t\ge 0.
\]
It is also stated (without proof) that one can show that
\[
\sigma_p(T_{w,a})=\left\{\lambda\in \C\,\middle|\, |\lambda|< |w| \right\}.
\]

Here, we completely describe the spectrum of such operators.

\begin{prop}[Spectrum]\label{C0S}\ \\
On the complex space $C_0[0,\infty)$, for the bounded linear weighted left translation operator
\begin{equation*}
(Tx)(t):=w x(t+a),\ t\ge 0,
\end{equation*}
where $w\in \C$ with $|w|>1$ and $a>0$,
\begin{equation}\label{spc0}
\sigma(T)=\left\{\lambda\in \C\,\middle|\, |\lambda|\le |w| \right\}
\end{equation}
with
\begin{equation}\label{pcsp0}
\sigma_p(T_{w,a})=\left\{\lambda\in \C\,\middle|\, |\lambda|< |w| \right\}\quad 
\text{and}\quad
\sigma_c(T_{w,a})=\left\{\lambda\in \C\,\middle|\, |\lambda|= |w| \right\}.
\end{equation}
\end{prop}

\begin{proof}
Let $w\in \C$ with  $w>1$ and $a>0$ be arbitrary and, for the simplicity of notation, let $T:=T_{w,a}$.

Since
\[
T=wB,
\]
where 
\begin{equation*}
(Bx)(t):=x(t+a),\ x\in C_0[0,\infty),\, t\ge 0,
\end{equation*}
is a left translation with $\|B\|=1$, and hence,
\begin{equation*}
\|T\|=|w|\|B\|=|w|,
\end{equation*} 
by \textit{Gelfand's Spectral Radius Theorem} \cite{Markin2020EOT},
\begin{equation}\label{spinclc0}
\sigma(T)\subseteq \left\{\lambda\in \C\,  \middle |\, |\lambda|\leq |w| \right\}.
\end{equation}

Let $\lambda\in \mathbb{C}$ with $|\lambda|<|w|$ and a nonzero $x\in C[0,a]$, with  
\[
x(a)=\dfrac{\lambda}{w}x(0)
\]
be arbitrary. E.g., for $0<|\lambda|<|w|$,
\[
y(t):=e^{ct},\ t\in [0,a],
\]
with $c:=\frac{1}{a}\ln\frac{\lambda}{w}=\frac{1}{a}\left(\ln\left|\frac{\lambda}{w}\right|+i\Imp\frac{\lambda}{w} \right)$ ($i$ is the \textit{imaginary unit}). 

Then, as is readily verified,
\begin{equation}\label{efboundedc0}
x_\lambda(t):=\left(\dfrac{\lambda}{w}\right)^{k}x(t-ka),\
t\in [ka,(k+1)a),k\in \Z_+,\ (0^0:=1)
\end{equation}
is a nonzero function continuous on $[0,\infty)$.

Since, in view of $|\lambda/w|<1$, for any $k\in \Z_+$,
\begin{equation*}
\max_{ka\le t\le(k+1)a}|x_\lambda(t)|=\max_{ka\le t\le(k+1)a} \left|\left(\dfrac{\lambda}{w}\right)^{k}x(t-ka)\right|
= \left|\dfrac{\lambda}{w}\right|^{k}\max_{0\le t\le a}\left|x(t) \right| \to 0,\, k\to\infty,
\end{equation*}
we infer that $x_\lambda\in C_{0}[0,\infty)\setminus \left\{ 0\right\}$. 

Also,
\begin{align*}
(Tx_\lambda)(t)&=wx(t+a)=w\left(\dfrac{\lambda}{w}\right)^{k}x(t+a-ka)
\\
&=\lambda\left(\dfrac{\lambda}{w}\right)^{k-1}x(t-(k-1)a),\ t\in [(k-1)a,ka),k\in \N,
\end{align*}
which implies that
\begin{equation}\label{EV1}
Tx_\lambda=\lambda x_\lambda.
\end{equation}

Thus, $\lambda\in \sigma_p(T)$.

Conversely, let $\lambda\in \sigma_p(T)$ be an arbitrary eigenvalue for $T$ with an associated eigenvector $x_{\lambda}\in C_0[0,\infty)\setminus \left\{0 \right\}$. Then, for
\[
x_k(t):=x_{\lambda}(t),\ t\in [ka,(k+1)a), k\in \Z_+,
\]
by \eqref{EV1}, we have:
\[
\lambda x_{k-1}(t)=wx_{k}(t+a),\ t\in [ka,(k+1)a), k\in \N.
\]
Whence,
\[
x_k(t)=\left(\dfrac{\lambda}{w}\right)^{k}x_{\lambda}(t-ka),\ t\in [ka,(k+1)a), k\in \Z_+,
\]
which, in view of $x_{\lambda}\ne 0$, implies that
\[
0<\max_{0\le t\le a}|x_{\lambda}(t)| \le \sup_{t\ge 0}|x_{\lambda}(t)|<\infty.
\]

Since
\[
\lim_{t\to \infty}x_\lambda(t)=0
\]
implies
\[
{\left|\dfrac{\lambda}{w}\right|}^k\max_{0\le t\le a}|x_{\lambda}(t)|\to 0,\ k\to\infty,
\]
which, in its turn, means that 
\[
|\lambda|<|w|.
\]

Thus, $x_\lambda$ is of the form given by \eqref{efboundedc0}, where $x$ is the restriction to $[0,a]$ of $x_\lambda$.

The above proves that
\begin{equation}\label{pspc0}
\sigma_p(T)=\left\{\lambda\in \C\,  \middle |\, |\lambda|<|w| \right\}.
\end{equation}

Considering that $\sigma(T)$ is a \textit{closed} set in $\C$ (see, e.g., \cite{Dun-SchI,Markin2020EOT}), we infer from \eqref{spinclc0} and \eqref{pspc0} that \eqref{spc0} holds.

Since, by \cite[Lemma $2.53$]{Grosse-Erdmann-Manguillot}, the hypercyclicity of $T$ implies the operator $T-\lambda I$ has a \textit{dense range} for all $\lambda\in \C$, we infer that 
\[
\sigma_{r}(T)=\emptyset
\]
(cf. \cite[Proposition $4.1$]{arXiv:2106.14872}, \cite[Lemma $1$]{MarkSich2019(1)}), and hence, in view of \eqref{spc0} and \eqref{pspc0}, we conclude that
\[
\sigma_c(T)=\left\{\lambda\in \C\,  \middle |\, |\lambda|=|w| \right\}.
\] 

Thus, \eqref{pcsp0} holds as well.
\end{proof}

\section{Unbounded Weighted Translations in $C_0[0,\infty)$}

\begin{lem}[Closedness of Powers]\label{CPC0}\ \\
In the (real or complex) space $\left(C_0[0,\infty),\|\cdot\|_\infty\right)$, for the weighted left translation
\begin{equation*}
(T_{w,a}x)(t):=w^{t} x(t+a),\ t\ge 0,
\end{equation*}
with $w>1$, $a>0$, and domain 
\begin{equation*}
D(T_{w,a}):=\left\{ x\in C_{0}[0,\infty) \middle |  \lim_{t\to \infty} w^{t} x(t+a) =0 \right\},
\end{equation*}
each power $T_{w,a}^n$ ($n\in \N$) is a densely defined unbounded closed linear operator.
\end{lem}

\begin{proof}
Let $w>1$, $a>0$, and $n\in \N$ be arbitrary and, for the simplicity of notation, let $T:=T_{w,a}$.

The \textit{linearity} of $T$ is obvious and implies that for $T^n$.

Inductively,
\begin{equation}\label{Tn2}
(T^nx)(t)=w^tw^{t+a}\dots w^{t+(n-1)a}x(t+na)
=w^{nt+\frac{(n-1)na}{2}}x(t+na),\ t\ge 0,
\end{equation}
and
\begin{equation}\label{DTn2}
D(T^n)=\left\{ x\in C_{0}[0,\infty)\, \middle |\,  \lim_{t\to \infty} w^{nt+\frac{(n-1)na}{2}}x(t+na)=0 \right\}
\end{equation}
(cf. \eqref{Tn} and \eqref{DTn}).

By the denseness in $C_0[0,\infty)$ ($1\le p<\infty$) of the subspace 
\begin{equation}\label{YC0U}
Y:=\bigcup_{m=1}^\infty \ker T^m,
\end{equation}
where
\begin{equation}\label{kerTnC0U}
\ker T^m=\left\{f\in C_0[0,\infty)\,\middle|\,  f(t)=0,\ t\ge ma\right\},\ m\in \N,
\end{equation}
of the equivalence classes represented by $p$-integrable on $(0,\infty)$ eventually zero functions and the inclusion
\begin{equation}\label{YC0Cinf}
Y\subset C^\infty(T):=\bigcap_{m=1}^\infty D(T^m),
\end{equation}
which follows from \eqref{DTn2}, we infer that the operator $T^n$ is \textit{densely defined}.

The  \textit{unboundedness} of $T^n$ follows from the fact that, for 
\[
e_m(t):=
\begin{cases}
1,& 0\le t<ma,\\
w^{-(t-ma)^2}, & t\ge ma,
\end{cases}
\quad m\in \N
\] 
we have:
\[
e_n\in D(T),\ \|e_n\|_{\infty}=1,\ m\in \N,
\]
and, for all $m\ge n$, in view of $w>1$, 
\begin{align*}
\|T^ne_m\|_\infty&=\sup_{t\ge 0}\left|w^{nt+\frac{(n-1)na}{2}}e_m(t+na) \right|
\\
&\ge w^{nt+\frac{(n-1)na}{2}}e_m(t+na)\bigr|_{t=ma-na}
\\
&\ge w^{n(ma-na)+\frac{(n-1)na}{2}}\to \infty,\ m\to \infty.
\end{align*}

Let a sequence $(x_m)_{m\in \N}$ in $C_{0}[0,\infty)$ be such that 
\[
D(T^n)\ni x_m\to x\in C_0[0,\infty),\ m\to \infty,
\]
and
\[
T^nx_m\to y\in C_0[0,\infty),\ m\to \infty.
\]

Then, for each $t\ge 0$,
\begin{equation}\label{pw}
x_m(t)\to x(t)\ \text{and}\
(T^nx_m)(t)\to y(t),\ m\to \infty.
\end{equation}

By \eqref{pw}, for each $t\ge 0$,
\[
(T^nx_m)(t)=w^{nt+\frac{(n-1)na}{2}}x_{m}(t+ma)\to w^{nt+\frac{(n-1)na}{2}}x(t+na),\ m\to \infty,
\]
and
\[
w^{nt+\frac{(n-1)na}{2}}x(t+na)=y(t),\ t\ge 0,
\]
which implies
\[
x\in D(T^n)\quad \text{and}\quad T^nx=y.
\]

Thus, by the \textit{Sequential Characterization of Closed Linear Operators} (see, e.g., \cite{Markin2020EOT}) the operator $T^n$ is \textit{closed}.
\end{proof}

\begin{thm}[Unbounded Weighted Translations in $C_{0}[0,\infty)$]\label{C0U}\ \\
In the (real or complex) space $\left(C_0[0,\infty),\|\cdot\|_\infty\right)$, the weighted left translation
\begin{equation*}
(T_{w,a}x)(t):=w^{t} x(t+a),\ t\ge 0,
\end{equation*}
with $w>1$, $a>0$, and domain 
\begin{equation*}
 D(T_{w,a}):=\left\{ x\in C_0[0,\infty) \middle |  \lim_{t\to \infty} w^{t} x(t+a) =0 \right\}
\end{equation*}
is a chaotic unbounded linear operator. 

Furthermore, provided the underlying space is complex,
\begin{equation}\label{sp2}
\sigma(T_{w,a})=\sigma_p(T_{w,a})=\mathbb{C}.
\end{equation}
\end{thm}
    
\begin{proof}
Let $w>1$ and $a>0$ be arbitrary and, for the simplicity of notation, let $T:=T_{w,a}$.

For the \textit{dense} in $C_0[0,\infty)$ subspace $Y$ of eventually zero functions (see \eqref{YC0U} and \eqref{kerTnC0U}), we have inclusion \eqref{YC0Cinf}.

The mapping 
\[
Y\ni x\mapsto Sx\in Y,
\]
where
\begin{equation}\label{S2}
(Sx)(t):=\begin{cases}
\frac{x(0)}{a}t,  &0\le t<a,\\
w^{-(t-a)}x(t-a), &t\geq a,\\ 
\end{cases}
\end{equation}
is well defined since the function $Sx$ is eventually zero and, as is easily seen,
\begin{equation}\label{RI2}
\forall\, x\in Y:\ TSx=x.
\end{equation}

Let $x\in  Y$ be arbitrary. Then 
\[
\exists\, M\in \mathbb{N}:\ \supp x:=\bar{\left\{t\in [0,\infty)\,\middle|\, x(t)\neq 0\right\}}\subseteq [0,Ma].
\]

By \eqref{Tn2},
\begin{equation*}
\forall\, n\ge M:\ T^nx=0,
\end{equation*}
and hence,
\[
T^nx\to 0,\ n\to \infty.
\]

Based on \eqref{S2}, inductively,
\begin{equation}\label{Sn2}
\begin{aligned}
&(S^nx)(t)
\\
&=
\begin{cases}
0,&0\le t <(n-1)a,\\
w^{-(t-a)}\dots w^{-(t-(n-1)a)}\frac{x(0)}{a}(t-(n-1)a), &(n-1)a\le t<na,\\
w^{-(t-a)}\dots w^{-(t-na)}x(t-na), & t\geq na,
\end{cases}
\\
&=
\begin{cases}
0,&0\le t <(n-1)a,\\
w^{-(t-a)}\frac{x(0)}{a}(t-(n-1)a), &(n-1)a\le t<na,\\
w^{-nt+\frac{n(n+1)a}{2}}x(t-na), & t\geq na,
\end{cases}
\quad x\in Y,n\in \N.
\end{aligned}
\end{equation}

In view of $w>1$,
\begin{equation*}
\|S^nx\|_{\infty}=\sup_{t\ge 0} \left|(S^nx)(t)\right|
\le w^{ -n\cdot na+\frac{n(n+1)a}{2}} \|x\|_{\infty}
=w^{-\frac{n(n-1)a}{2}}\|x\|_{\infty},\ x\in Y,n\in \N.
\end{equation*}

Whence, since $w>1$ and $a>0$, we deduce that 
\begin{equation*}
\forall\, x\in Y:\ \lim_{n\to \infty}{\|S^nx\|_\infty}^{1/n}=0,
\end{equation*}
or equivalently,
\begin{equation}\label{GS2}
\forall\, x\in Y,\ \forall\, \alpha\in (0,1)\
\exists\, c=c(x,\alpha)>0\ \forall\, n\in \N:\ \|S^nx\|_\infty\le c\alpha^n \|x\|_\infty,
\end{equation}
which implies
\[
\forall\, x\in Y:\ S^nx\to 0,\ n\to \infty.
\]

From the above and the fact that, by the \textit{Closedness of Powers Lemma} (Lemma \ref{CPC0}), each power $T^n$ ($n\in \N$) is a \textit{closed operator}, by the \textit{Sufficient Condition for Hypercyclicity} (Theorem \ref{SCH}), we infer that the operator $T$ is  \textit{hypercyclic}.

Based on estimate \eqref{GS2}, proving that $T$ has a dense set of periodic points, and hence, is \textit{chaotic} and that \eqref{sp2} holds is identical to proving the same parts in Theorem \ref{LpU}.
\end{proof}

\section{Concluding Remarks}

The foregoing results are consistent with the recent findings of 
\cite{arXiv:2106.14872}. According to the latter,
under the premises of Theorem \ref{Lp}, Theorem \ref{LpU}, \cite[Theorem $2.3$]{Aron-Seoan-Weber2005}, or Theorem \ref{C0U}, not only 
is the operator $T_{w,a}$ \textit{chaotic} but also its every power $T_{w,a}^n$ ($n\in \N$) and, furthermore, 
\[
\dim\ker\left(T_{w,a}^n-\lambda I\right)=\dim\ker T_{w,a}^n
=\dim\left\{f\in X\,\middle|\,  f(t)=0,\ t>na\right\},
\]
where $X:=L_p(0,\infty)$ ($1\le p<\infty$) or $X:=C_0[0,\infty)$, holds in Theorem \ref{Lp} and Proposition \ref{C0S} for all $n\in \N$ and $\lambda\in \C$ with $|\lambda|<{|w|}^n$ and in Theorems \ref{LpU} and \ref{C0U} for all $n\in \N$ and $\lambda\in \C$, i.e., all eigenvalues of $T_{w,a}^n$ are of the same geometric multiplicity.




\end{document}